\documentclass[11pt]{amsart}
\usepackage[margin=1.15in]{geometry}
\usepackage{amssymb,mathtools}
\usepackage{booktabs}
\usepackage{enumitem}
\usepackage{microtype}
\usepackage{xcolor}
\usepackage[colorlinks=true,linkcolor=blue!50!black,citecolor=blue!50!black,urlcolor=blue!50!black]{hyperref}

\newtheorem{theorem}{Theorem}[section]
\newtheorem{lemma}[theorem]{Lemma}
\newtheorem{proposition}[theorem]{Proposition}
\newtheorem{corollary}[theorem]{Corollary}
\theoremstyle{definition}
\newtheorem{definition}[theorem]{Definition}
\newtheorem{construction}[theorem]{Construction}
\theoremstyle{remark}
\newtheorem{remark}[theorem]{Remark}
\numberwithin{equation}{section}

\newcommand{\F}{\mathbb{F}}
\newcommand{\Z}{\mathbb{Z}}
\newcommand{\Q}{\mathbb{Q}}
\newcommand{\C}{\mathbb{C}}
\newcommand{\eps}{\varepsilon}
\newcommand{\PG}{\mathrm{PG}}
\newcommand{\Mbar}{\overline{M}}
\DeclareMathOperator{\girth}{girth}

\title[Zero divisors in odd characteristic]{Zero divisors in group algebras of torsion-free groups in odd characteristic}

\author{S.V. Ivanov}
\address{Department of Mathematics\\
University of Illinois \\
Urbana\\   IL 61801\\ U.S.A.}
  
\date{October 7, 2026}

\begin{document}

\begin{abstract}
A recent preprint  \cite{OAI26} constructs a finitely presented torsion-free group $G$ with a finite two-dimensional classifying space such that $\F_2[G]$ contains nonzero zero divisors. We extend this to every odd prime $p$: there is a finitely presented torsion-free group $G_p$, again with a finite two-dimensional classifying space, such that $\F_p[G_p]$ contains nonzero elements $\alpha,\beta$ with $\alpha\beta=0$. Together with the case $p=2$, Kaplansky's zero-divisor conjecture therefore fails over every field of positive characteristic. The new ingredients are a signed version of the parity cancellation, made possible by a bipartite random model on one side; a type design built from the projective plane over $\F_{p^3}$ with $p$ additional letters, which makes every cross-intersection congruent to $1$ modulo $p$; and a squared-word decay estimate for this design, proved with an explicit test vector. The geometric part of the argument is taken from the characteristic-two construction. We record a line-by-line check of that part, together with the bookkeeping changes it needs. The characteristic-zero case remains open.
\end{abstract}

\maketitle

\section{Introduction}

Let $K$ be a field and $G$ a group. Kaplansky's zero-divisor conjecture asserts that the group algebra $K[G]$ has no nonzero zero divisors whenever $G$ is torsion-free. The question goes back to Higman's 1940 thesis \cite{Hig40a}, whose relevant part is reproduced by Sandling \cite{San81}, and it is Problem~6 in Kaplansky's 1956 list \cite{Kap57}. Torsion is a genuine obstruction: if $g$ has order $m>1$, then $(1-g)(1+g+\dots+g^{m-1})=0$.

The conjecture is known for large classes of groups. Higman proved it for locally indicable groups \cite{Hig40b}. Kropholler, Linnell and Moody proved it for torsion-free elementary amenable groups \cite{KLM88}, and Fisher and S\'anchez-Peralta recently treated torsion-free virtually compact special groups and torsion-free $3$-manifold groups \cite{FSP26}. Groups with the unique-product property also satisfy it, but torsion-free groups without that property have been known since Rips and Segev \cite{RS87}; see also \cite{Ste15,MS26}.

A recent preprint \cite{OAI26} disproves the conjecture over $\F_2$. It constructs, by a probabilistic method, a finitely presented torsion-free group $G$ admitting a finite two-dimensional classifying space, together with nonzero $\alpha,\beta\in\F_2[G]$ such that $\alpha\beta=0$. The cancellation in that construction is specific to characteristic two. Our main result removes this restriction.

\begin{theorem}\label{thm:main}
Let $p$ be an odd prime. There exist a finitely presented torsion-free group $G_p$, admitting a finite two-dimensional classifying space, and nonzero elements $\alpha,\beta\in\F_p[G_p]$ such that $\alpha\beta=0$.
\end{theorem}

Since $\F_p[G_p]\subseteq K[G_p]$ for every field $K$ of characteristic $p$, the same $\alpha$ and $\beta$ are zero divisors in $K[G_p]$. Combining this with \cite[Theorem~1.1]{OAI26} gives the following.

\begin{corollary}\label{cor:allp}
For every field $K$ of positive characteristic there is a finitely presented torsion-free group $G$, admitting a finite two-dimensional classifying space, such that $K[G]$ has nonzero zero divisors.
\end{corollary}

\subsection{The obstruction in odd characteristic}
We recall the mechanism of \cite{OAI26}. Two finite labelled graphs $\Gamma_A,\Gamma_B$ are immersed in a rose, and a cone is attached along each of their components. In the resulting group, a path from a root $x_A$ to a vertex $x$ has a well-defined label $g_x$, and similarly a path from $x_B$ to $y$ has a label $h_y$. One sets $\alpha=\sum_x g_x$ and $\beta=\sum_y h_y^{-1}$. The \emph{step graph} has vertex set $A'\times B'$, where $A',B'$ are the components of the roots, and an edge whenever $x$ and $y$ leave along a common letter $t$. Since $(g_xt)(h_yt)^{-1}=g_xh_y^{-1}$, the product $g_xh_y^{-1}$ is constant on each component $C$ of the step graph, and
\begin{equation}\label{eq:intro-sum}
\alpha\beta=\sum_{C}\lvert C\rvert\,g_C .
\end{equation}
The degree of $(x,y)$ is the number $\lvert S_x\cap S_y\rvert$ of common outgoing letters. In \cite{OAI26} all these numbers are odd, so the handshake identity makes every $\lvert C\rvert$ even and \eqref{eq:intro-sum} vanishes over $\F_2$.

For an odd prime $p$ this argument gives nothing. The handshake identity only says that the degree sum is twice the number of edges. If all degrees are congruent to $d$ modulo $p$, this yields $d\lvert C\rvert\equiv 2\lvert E(C)\rvert$, which does not constrain $\lvert C\rvert$ modulo $p$.

\subsection{Signed cancellation}
Our remedy is to put signs on $\alpha$. Suppose the component $A'$ is bipartite with a sign $\eps\colon A'\to\{\pm1\}$ that alternates along edges, and put $\alpha=\sum_x\eps(x)g_x$. Every step edge then joins vertices of opposite sign, so in the signed degree sum $\sum_{(x,y)\in C}\eps(x)\deg(x,y)$ the two ends of each edge cancel exactly. If moreover every degree is congruent to a fixed $d\not\equiv0$ modulo $p$, the signed vertex count of each component vanishes modulo $p$, and so does $\alpha\beta$. This is Lemma~\ref{lem:signed}. For $p=2$ it reduces to the argument of \cite{OAI26}, because $-1=1$.

\subsection{The construction}
To meet the hypotheses of Lemma~\ref{lem:signed} we change three things in \cite{OAI26}.
\begin{enumerate}[label=(\roman*)]
\item The alphabet consists of the points of the projective plane $\PG(2,q)$ with $q=p^3$, together with $p$ extra letters. A vertex type is a line, possibly together with all $p$ extra letters. Since two lines meet in $1$ or $q+1$ points, every cross-intersection is $\equiv1\pmod p$ (Lemma~\ref{lem:residue}).
\item The graph $\Gamma_A$ is sampled as a random bipartite graph between two halves $A^+$ and $A^-$, which supplies the signs; $\Gamma_B$ is sampled as in \cite{OAI26} (Section~\ref{sec:random}).
\item The squared-word decay estimate of \cite[Lemma~2.1]{OAI26} is reproved for the new alphabet (Proposition~\ref{prop:decay}). Beyond its degree and frequency parameters, this is the only place where the specific design enters the probabilistic estimates, and it is where the exponent $3$ in $q=p^3$ is used.
\end{enumerate}
The geometric part of \cite{OAI26} --- the bounded-pattern estimate, the reduction of spherical arrangements to bounded patterns, and the cone-picture surgery --- then applies with one bounded extra factor in a realization count (Section~\ref{sec:transfer}).

\subsection{Dependence on \cite{OAI26}}
Our proof uses Sections~3--5 of \cite{OAI26} with the modifications just described, and we do not reproduce those sections. We have checked them, together with the lemmas of Section~2 of \cite{OAI26} on which they rely, step by step. We found no error affecting the conclusions; Appendix~\ref{app:verification} records six points that deserve a clarifying sentence. 

\subsection{Characteristic zero}
The method says nothing over $\Q$. Over $\Z$, Lemma~\ref{lem:signed} gives $d\sum_{(x,y)\in C}\eps(x)=0$ only if every degree equals $d$ exactly, not merely modulo $p$. The direct way to achieve this would be a design in which every type on one side meets every type on the other side in the same number of letters. Projective lines do not do this: their intersection matrix is $qI+J$, which has rank one only modulo $p$. A rational counterexample would in addition contradict the strong Atiyah conjecture for the group in question, since that conjecture implies the zero-divisor conjecture over $\C$ for torsion-free groups \cite{Lin93,Lue02}. We discuss this further in Section~\ref{sec:remarks}.

\subsection{Related work}
Gardam disproved Kaplansky's unit conjecture over $\F_2$ \cite{Gar21} and later over $\C$ \cite{Gar24}; a nontrivial unit is not a zero divisor. Mineyev introduced product structures giving sufficient geometric criteria for nontrivial units and zero divisors over $\F_2$ \cite{Min24}, studied computationally by Garg and Mineyev \cite{GM25}; Shin found a girth obstruction for one such route \cite{Shi26}. The probabilistic and topological methods of \cite{OAI26}, which we inherit, build on ideas from graphical small cancellation \cite{Gro03,Oll05,Oll06,Gru15,Ber19} and on planar separation \cite{LT79}.

\subsection{Conventions and organization}
We follow the conventions of \cite{OAI26}. All graphs are finite. A signed alphabet $T$ carries a fixed-point-free involution $t\mapsto\bar t$, and the rose $F$ has one vertex and one edge for each pair $\{t,\bar t\}$. A labelled graph is \emph{immersed} if at each vertex every letter labels at most one outgoing edge; $S_x$ denotes the set of outgoing letters at $x$. Lengths count edge traversals with repetition, logarithms are natural, and constants are independent of the sampling size $n$ unless stated otherwise. References of the form ``\cite[Lemma~2.2]{OAI26}'' point to the source preprint.

Section~\ref{sec:signed} proves the signed cancellation lemma. Section~\ref{sec:types} constructs the alphabet and the types, Section~\ref{sec:random} the random model, and Section~\ref{sec:decay} the decay estimate. Section~\ref{sec:transfer} explains how the arguments of \cite[Sections~3--5]{OAI26} transfer, and Section~\ref{sec:proof} proves Theorem~\ref{thm:main}. Section~\ref{sec:remarks} collects further remarks. Appendix~\ref{app:verification} records our check of \cite{OAI26}, and Appendix~\ref{app:computations} the computer verifications.

\section{Signed cancellation}\label{sec:signed}

Throughout this section $T$ is a signed alphabet and $\mathbb{F}(T)$ the free group on one letter from each pair $\{t,\bar t\}$, with $\bar t=t^{-1}$. A path in a $T$-labelled graph reads a word, hence an element of $\mathbb{F}(T)$.

\begin{definition}\label{def:setting}
Let $G$ be a group and $\phi\colon\mathbb{F}(T)\to G$ a homomorphism. Let $\Gamma_A,\Gamma_B$ be immersed $T$-labelled graphs, let $x_A,x_B$ be vertices of $\Gamma_A,\Gamma_B$, and let $A',B'$ be the vertex sets of their components. Assume that $\phi$ kills the label of every closed walk in these two components. For $x\in A'$ and $y\in B'$, let $g_x,h_y\in G$ be the images of the labels of any paths from $x_A$ to $x$ and from $x_B$ to $y$; they do not depend on the paths.

The \emph{step graph} $\Sigma$ has vertex set $A'\times B'$. A \emph{dart} at $(x,y)$ is a pair $((x,y),t)$ with $t\in S_x\cap S_y$. It leads to $(x',y')$, where $x'$ and $y'$ are the ends of the edges labelled $t$ at $x$ and $y$; the darts $((x,y),t)$ and $((x',y'),\bar t)$ form one edge of $\Sigma$.
\end{definition}

Because $A'$ and $B'$ are full components, every dart at a vertex of $A'\times B'$ leads back into $A'\times B'$, and the degree of $(x,y)$ in $\Sigma$ is exactly $\lvert S_x\cap S_y\rvert$.

\begin{lemma}[Signed cancellation]\label{lem:signed}
In the setting of Definition~\ref{def:setting}, let $p$ be a prime and suppose:
\begin{enumerate}[label=\textup{(\roman*)}]
\item the component $A'$ is bipartite, with a sign $\eps\colon A'\to\{\pm1\}$ such that $\eps(x')=-\eps(x)$ whenever $x$ and $x'$ are adjacent;
\item there is an integer $d$ with $p\nmid d$ such that $\lvert S_x\cap S_y\rvert\equiv d\pmod p$ for all $x\in A'$ and $y\in B'$.
\end{enumerate}
Put
\[
\alpha=\sum_{x\in A'}\eps(x)\,g_x,\qquad \beta=\sum_{y\in B'}h_y^{-1}.
\]
Then $\alpha\beta=0$ in $\F_p[G]$.
\end{lemma}

\begin{proof}
\emph{Step 1: constancy.} If a dart at $(x,y)$ with letter $t$ leads to $(x',y')$, then $g_{x'}=g_x\phi(t)$ and $h_{y'}=h_y\phi(t)$, so
\[
g_{x'}h_{y'}^{-1}=g_x\phi(t)\phi(t)^{-1}h_y^{-1}=g_xh_y^{-1}.
\]
Hence $g_xh_y^{-1}$ is constant on each connected component $C$ of $\Sigma$; denote its value by $g_C$.

\emph{Step 2: darts pair up.} Sending a dart $((x,y),t)$ to $((x',y'),\bar t)$ is an involution on the set of darts. It has no fixed points, since $\bar t\neq t$. Because $x$ and $x'$ are adjacent in $\Gamma_A$, hypothesis (i) gives $\eps(x')=-\eps(x)$.

\emph{Step 3: signed handshake.} Fix a component $C$. Summing $\eps$ over the initial vertices of all darts in $C$, and grouping the darts first by initial vertex and then by the pairs of Step~2, we get
\[
\sum_{(x,y)\in C}\eps(x)\,\lvert S_x\cap S_y\rvert
=\sum_{\text{dart pairs in }C}\bigl(\eps(x)+\eps(x')\bigr)=0 .
\]
This is an identity in $\Z$.

\emph{Step 4: reduction modulo $p$.} By hypothesis (ii), the left-hand side is congruent modulo $p$ to $d$ times the signed count $\sum_{(x,y)\in C}\eps(x)$. Since $p\nmid d$, the signed count of every component is divisible by $p$. Therefore
\[
\alpha\beta=\sum_{(x,y)\in A'\times B'}\eps(x)\,g_xh_y^{-1}
=\sum_{C}\Bigl(\sum_{(x,y)\in C}\eps(x)\Bigr)g_C=0\quad\text{in }\F_p[G].
\]
Different components may share the same value $g_C$; this only adds zero coefficients.
\end{proof}

\begin{remark}\label{rem:signed}
\begin{enumerate}[label=(\alph*)]
\item The step graph has no loops: the two darts of an edge sit at vertices with opposite signs. In \cite{OAI26} loops were excluded by girth; here bipartiteness excludes them automatically.
\item Only $\Gamma_A$ carries signs. The graph $\Gamma_B$ need not be bipartite.
\item For $p=2$ the signs are invisible and hypothesis (ii) says that every cross-intersection is odd. Lemma~\ref{lem:signed} then reduces to the cancellation of \cite[Section~6]{OAI26}.
\item The identity of Step~3 holds over $\Z$, but Step~4 only gives divisibility by $p$. The signed counts need not vanish; see Appendix~\ref{app:computations} for an example.
\end{enumerate}
\end{remark}

Nonvanishing of $\alpha$ and $\beta$ is a separate matter. If no path from $x_A$ to a different vertex of $A'$ has trivial image in $G$, then $x_A$ is the only contributor of the identity to $\alpha$, whose identity coefficient is therefore $\eps(x_A)=\pm1$. The same applies to $\beta$, whose identity coefficient is then $1$.

\section{The alphabet and the vertex types}\label{sec:types}

Fix an odd prime $p$ and put
\begin{equation}\label{eq:params}
q=p^{3},\qquad v=q^{2}+q+1,\qquad f=\frac{q+1}{v}.
\end{equation}
The projective plane $\PG(2,q)$ has $v$ points and $v$ lines; each line has $q+1$ points, each point lies on $q+1$ lines, and two distinct lines meet in exactly one point. The exponent $3$ is needed only in Section~\ref{sec:decay}; the constructions and lemmas of this section work for any power $q$ of $p$.

\begin{construction}[Alphabet]\label{con:alphabet}
Let $P$ be the set of points of $\PG(2,q)$ (the \emph{ordinary letters}) and $E=\{e_1,\dots,e_p\}$ a set of $p$ \emph{extra letters}. Choose $p$ distinct \emph{special points} $\sigma_1,\dots,\sigma_p\in P$ and put $\bar e_i=\sigma_i$. Since $v$ and $p$ are both odd, the remaining $v-p$ points can be paired among themselves arbitrarily. This defines a fixed-point-free involution on $T=P\sqcup E$, with $\lvert T\rvert=v+p$.
\end{construction}

An odd number of extras is needed because $v$ is odd, and the extras must meet each other in a multiple of $p$ so as not to disturb the residue condition. Taking exactly $p$ extras, used all-or-nothing, achieves both and allows the same design on both sides. For $p=2$ this is impossible: an odd number of extras is forced, and an odd number is not a multiple of $2$. This is consistent with the asymmetric design of \cite{OAI26}, which uses pairs of extras on one side and triples on the other so that extra parts meet in $0$ or $2$ letters.

\begin{construction}[Types]\label{con:types}
Let $n\in2v^2\Z$. Let $A=A^+\sqcup A^-$ with $\lvert A^\pm\rvert=n/2$, and let $\lvert B\rvert=n$. To each vertex $x$ we assign a type $S_x\subseteq T$ as follows. On $B$, and separately on each of $A^+$ and $A^-$:
\begin{enumerate}[label=(\roman*)]
\item the ordinary part $S_x\cap P$ is a line, and every line occurs equally often;
\item within each line class, a proportion $f$ of the vertices have $S_x\cap E=E$, and the others have $S_x\cap E=\emptyset$.
\end{enumerate}
\end{construction}

For $Y\in\{A^+,A^-,B\}$, each line class has $\lvert Y\rvert/v$ vertices, of which $\lvert Y\rvert(q+1)/v^2$ carry the extras; both numbers are integers because $n\in 2v^2\Z$. Table~\ref{tab:counts} gives the counts for $n=2v^2$. Vertex degrees are $q+1$ or $q+1+p$. For $p=3$ we have $v=757$, $\lvert T\rvert=760$, and degrees $28$ or $31$.

\begin{table}[ht]
\centering
\begin{tabular}{@{}lccc@{}}
\toprule
Vertex set & Size & Per line class & Of which carry $E$\\
\midrule
$A^+$ (likewise $A^-$) & $v^2$ & $v$ & $q+1$\\
$B$ & $2v^2$ & $2v$ & $2(q+1)$\\
\bottomrule
\end{tabular}
\caption{Type counts at $n=2v^2$.}\label{tab:counts}
\end{table}

For a vertex set $Y$ and a letter $t$, write $Y_t=\{x\in Y: t\in S_x\}$.

\begin{lemma}[Frequencies]\label{lem:freq}
For $Y\in\{A^+,A^-,B\}$ and every $t\in T$, we have $\lvert Y_t\rvert=f\lvert Y\rvert$.
\end{lemma}

\begin{proof}
A point lies on $q+1$ of the $v$ lines, and the extras are assigned independently of the line, so an ordinary letter lies in a proportion $(q+1)/v=f$ of the types. An extra letter lies exactly in the types carrying $E$, which form a proportion $f$ of every line class.
\end{proof}

\begin{lemma}[Residue]\label{lem:residue}
For all $x\in A$ and $y\in B$, we have $\lvert S_x\cap S_y\rvert\equiv1\pmod p$.
\end{lemma}

\begin{proof}
The ordinary parts are two lines, which meet in $1$ point if they are distinct and in $q+1$ points if they are equal. Both numbers are $\equiv1\pmod p$ because $p\mid q$. The extra parts meet in $p$ letters if both types carry $E$ and in no letter otherwise; both numbers are $\equiv0\pmod p$.
\end{proof}

Lemma~\ref{lem:residue} holds for every choice of matchings. As in \cite{OAI26}, the residue condition is built into the types, and the random choices below only have to control the geometry. In matrix language, the line--line intersection matrix of $\PG(2,q)$ is $qI+J$, which has rank one modulo $p$ but full rank over $\Q$; this is the structural reason the construction lives in characteristic $p$.

\section{The random model}\label{sec:random}

\begin{construction}[Matchings]\label{con:matchings}
For each inverse pair $\{t,\bar t\}$ choose, independently of all other choices:
\begin{enumerate}[label=(\roman*)]
\item a uniform bijection $B_t\to B_{\bar t}$;
\item two uniform bijections $A^+_t\to A^-_{\bar t}$ and $A^-_t\to A^+_{\bar t}$.
\end{enumerate}
A matched pair $(x,y)$ becomes an edge labelled $t$ at $x$ and $\bar t$ at $y$. Let $\Gamma_A$ and $\Gamma_B$ be the resulting graphs on $A$ and $B$, and $\Gamma=\Gamma_A\sqcup\Gamma_B$.
\end{construction}

By Lemma~\ref{lem:freq} all these bijections exist. We call each of them a \emph{sub-matching}; it has $N_B=nf$ pairs on $B$ and $N_A=nf/2$ pairs on $A$. Every vertex receives exactly one outgoing edge for each letter of its type, so $\Gamma$ immerses into the rose $F$ on $T$. Every edge of $\Gamma_A$ joins $A^+$ to $A^-$, so $\Gamma_A$ is bipartite and
\[
\eps(x)=\pm1\quad\text{for }x\in A^{\pm}
\]
satisfies hypothesis (i) of Lemma~\ref{lem:signed}. The maximum degree is $d_*=q+1+p$.

Fix a constant $c>0$ with
\begin{equation}\label{eq:c}
c\log\Bigl(\frac{4\lvert T\rvert}{f}\Bigr)<1,\qquad 2c\log d_*<1,
\end{equation}
and put $L=\lfloor c\log n\rfloor$. From now on all probabilities refer to the uniform distribution on tuples of matchings with $\girth(\Gamma)\ge L$, and $n$ runs through $2v^2\Z$. The factor $4$ in \eqref{eq:c} replaces the factor $2$ of \cite[(2.5)]{OAI26} and accounts for the halved sub-matchings on $A$.

\begin{lemma}[Switching]\label{lem:switch}
For all sufficiently large $n$, tuples with $\girth(\Gamma)\ge L$ exist. Condition on $\girth(\Gamma)\ge L$ and on any feasible set of $s$ prescribed pairs, and let $xy$ be a further pair of a sub-matching of size $N$, not prescribed in either direction. Then
\[
\Pr(xy\text{ present})\le\frac{1}{N-s-r_n},\qquad r_n=O\bigl(d_*^{\,2L+2}\bigr)=o(n),
\]
whenever the denominator is positive, with $r_n$ independent of the prescriptions. A pair joining two vertices of the same sign has probability $0$.
\end{lemma}

\begin{proof}
We follow \cite[Lemma~2.2]{OAI26}, working inside a single sub-matching throughout. Call an edge \emph{bad} if it lies on a cycle of length less than $L$.

\emph{Existence.} Before conditioning, a cycle of length $k$ has at most $2n^k\lvert T\rvert^k$ descriptions by vertices and letters. Its $k$ edges are distinct matching pairs, all present with probability at most $(nf/2-k)^{-k}$. For large $n$ the expected number of bad edges is therefore at most
\[
\sum_{k<L}2k\,n^k\lvert T\rvert^k\Bigl(\frac{nf}{4}\Bigr)^{-k}=2\sum_{k<L}k\Bigl(\frac{4\lvert T\rvert}{f}\Bigr)^k=O\bigl(L^2n^{c\log(4\lvert T\rvert/f)}\bigr)=o(n)
\]
by \eqref{eq:c}. Fix a tuple with $o(n)$ bad edges, let $xy$ be bad, oriented with label $t$ in a sub-matching $M$, and choose an edge $uz$ of $M$, also oriented with label $t$, which is not bad and has both ends at distance more than $2L$ from $\{x,y\}$. Such an edge exists because bad edges number $o(n)$ and the $2L$-neighbourhood of $\{x,y\}$ meets $O(d_*^{2L+2})=o(n)$ edges. Replace $xy,uz$ by $xz,uy$. Both new pairs lie in $M$, so the result is again a tuple of matchings, and on $A$ both new edges still join $A^+$ to $A^-$.

Let $Z$ be a cycle of length less than $L$ through a new edge. If $Z$ uses only $xz$, then $Z-xz$ is an old path from $x$ to $z$ of length less than $L$, contradicting the choice of $uz$; the same applies to $uy$. If $Z$ uses both new edges, the two arcs of $Z-\{xz,uy\}$ are old paths pairing $\{x,z\}$ with $\{u,y\}$, either as $(x,u),(z,y)$ or as $(x,y),(z,u)$. The first contradicts the distance condition. In the second, the arc from $z$ to $u$ together with the old edge $uz$ is a cycle of length less than $L$, contradicting the choice of $uz$. Hence no new edge is bad; $xy$ was bad and $uz$ was not, so the number of bad edges decreases. Repeating gives a tuple of girth at least $L$.

\emph{Conditional bound.} Let $\Omega_s$ be the set of girth-$\ge L$ tuples containing the $s$ prescribed pairs, and $\Omega_{s,xy}\subseteq\Omega_s$ those that also contain $xy$, oriented with label $t$. If $xy$ is incompatible with the types or the signs, $\Omega_{s,xy}=\emptyset$. Otherwise, for $\omega\in\Omega_{s,xy}$ let $U(\omega)$ be the set of edges $uz$ of the same sub-matching, oriented with label $t$, not prescribed, and with both ends at distance more than $2L$ from $\{x,y\}$. Then $\lvert U(\omega)\rvert\ge N-s-r_n$, where $r_n$ bounds the number of edges meeting the $2L$-neighbourhood of $\{x,y\}$ plus one; this bound depends only on $d_*$ and $L$. Every edge of $\omega$ is good, so the switching argument above applies and produces $\omega'\in\Omega_s$ with the prescriptions untouched. The map $(\omega,uz)\mapsto\omega'$ is injective: in $\omega'$ the $t$-neighbour of $x$ is $z$ and the $\bar t$-neighbour of $y$ is $u$, which recovers $uz$ and then $\omega$. Hence $\lvert\Omega_{s,xy}\rvert(N-s-r_n)\le\lvert\Omega_s\rvert$. Finally $r_n=O(d_*^{2L+2})=o(n)$ by \eqref{eq:c}.
\end{proof}

\begin{corollary}\label{cor:prescriptions}
Any feasible collection of $E=O(L)$ distinct prescriptions, $E_A$ on $A$ and $E_B$ on $B$, has conditional probability at most
\[
\Bigl(\frac{nf}{2}\Bigr)^{-E_A}(nf)^{-E_B}\,e^{o(L)},
\]
uniformly over such collections.
\end{corollary}

\begin{proof}
Apply Lemma~\ref{lem:switch} successively. The product of the factors $(1-(k+r_n)/N)^{-1}$ for $k<E$ is $\exp(O(E(E+r_n)/n))=e^{o(L)}$.
\end{proof}

\begin{lemma}[Diameter]\label{lem:diameter}
There is a constant $D_0$, depending only on $p$, such that with probability tending to one every component of $\Gamma$ has diameter at most $D_0L$.
\end{lemma}

\begin{proof}
We follow the proof of \cite[Lemma~2.3]{OAI26}, which bounds the probability directly by the switching lemma and allows the constant $\rho$ to be decreased as needed. We decrease $\rho$ so that, in addition, $\rho\le1/(4v)$; the final step of that proof, which bounds the diameter using balls of more than $\rho n$ vertices, works for every fixed $\rho>0$.

A $k$-set $S$, $1\le k\le\rho n$, that fails to double on adjoining its neighbours is contained in a set $W$ of $2k$ vertices containing all edges at $S$. Since every vertex has degree at least $q+1$, the set $W$ spans at least $r=\lceil(q+1)k/2\rceil$ distinct edges; fix $r$ of them.

The number $r$ grows linearly in $k$ and may be of order $n$, so Corollary~\ref{cor:prescriptions}, which concerns $O(L)$ prescriptions, does not apply. Instead we use Lemma~\ref{lem:switch} directly, prescribing the $r$ edges one at a time. When the $(s+1)$-st edge is added, at most
\[
s\le r-1\le\frac{(q+1)\rho n}{2}\le\frac{(q+1)n}{8v}=\frac{nf}{8}
\]
edges are prescribed, and the relevant sub-matching has size $N\ge N_A=nf/2$. Since $r_n=o(n)$, for large $n$ we have $s+r_n\le nf/4\le N/2$, so each factor supplied by Lemma~\ref{lem:switch} is at most $2/N\le4/(nf)$. Hence any fixed collection of $r$ distinct edges, with their letters, is present with conditional probability at most $(4/(nf))^r$ (and with probability $0$ if it is not feasible).

The set $W$ can be chosen in at most $\binom{2n}{2k}\le(en/k)^{2k}$ ways. An edge inside $W$ is determined by one of its ends, the letter at that end (at most $d_*$ choices) and its other end, so $W$ spans at most $4k^2d_*$ candidate edges, and $r$ of them can be chosen in at most $(4ek^2d_*/r)^r\le(16ek)^r$ ways, because $r\ge(q+1)k/2$ and $d_*\le2(q+1)$. Since $r\le(q+1)k/2+1\le(q+2)k$, the probability that some $k$-set fails to double is therefore at most
\[
\Bigl(\frac{en}{k}\Bigr)^{2k}(16ek)^r\Bigl(\frac{4}{nf}\Bigr)^r
=e^{2k}\Bigl(\frac{64e}{f}\Bigr)^r\Bigl(\frac{k}{n}\Bigr)^{r-2k}
\le\Bigl(C_2\Bigl(\frac kn\Bigr)^{r/k-2}\Bigr)^k,\qquad C_2=e^2\Bigl(\frac{64e}{f}\Bigr)^{q+2},
\]
and $r/k-2\ge(q+1)/2-2\ge12$. The constant $C_2$ depends only on $p$, not on $\rho$. Any positive exponent suffices for the summation over $k$; the rest of the argument is unchanged.
\end{proof}

\section{Turn weights and squared-word decay}\label{sec:decay}

This section replaces \cite[Lemma~2.1]{OAI26}. Apart from the degree and frequency parameters $d_*$, $f$ and the minimum turn weight, it is the only point where the specific type design enters the probabilistic estimates.

For letters $t,u$ with $u\neq\bar t$, let $w(t,u)$ be the proportion of vertices of $B$ whose type contains both $\bar t$ and $u$, among those whose type contains $\bar t$; put $w(t,\bar t)=0$. Since the type distributions on $A^+$, $A^-$ and $B$ coincide, the same proportions arise on $A^+$ and on $A^-$. For a word $W=t_1\cdots t_h$ put
\[
P(W)=\prod_{i=1}^{h-1}w(t_i,t_{i+1}).
\]

\begin{lemma}[Turn weights]\label{lem:weights}
For $u\neq\bar t$,
\[
w(t,u)=\begin{cases}1/(q+1)&\text{if }\bar t,u\in P,\\ f&\text{if exactly one of }\bar t,u\text{ lies in }E,\\ 1&\text{if }\bar t,u\in E.\end{cases}
\]
Moreover $w(t,u)=w(\bar u,\bar t)$, and hence $P(W)=P(W^{-1})$.
\end{lemma}

\begin{proof}
Two distinct points lie on exactly one common line, so exactly one of the $q+1$ lines through $\bar t$ contains $u$. The extras are assigned independently of the line, and a type containing one extra contains all of them. For the symmetry, $w(t,u)=\lvert B_{\bar t}\cap B_u\rvert/(f\lvert B\rvert)$ by Lemma~\ref{lem:freq}, and this expression is symmetric under $(t,u)\mapsto(\bar u,\bar t)$.
\end{proof}

In particular every non-cancelling turn has positive weight, and the smallest positive weight is $1/(q+1)$, which is slightly smaller than $f$. As in \cite{OAI26},
\begin{equation}\label{eq:transfer}
\sum_{W\in T^h}P(W)^2=\mathbf 1^{\top}M^{h-1}\mathbf 1,\qquad M_{t,u}=w(t,u)^2 .
\end{equation}

\subsection{Lumping}
Partition $T$ into three classes according to $\bar t$:
\[
\mathcal E=\{t:\bar t\in E\},\qquad \mathcal S=\{t:\bar t\text{ is special}\},\qquad \mathcal O=T\setminus(\mathcal E\cup\mathcal S),
\]
of sizes $p$, $p$ and $v-p$. The next state after $t$ is the next letter $u$, whose class is read off from $\bar u$: an extra $u$ lies in $\mathcal S$, a special point $u$ lies in $\mathcal E$, and any other point $u$ lies in $\mathcal O$. Put
\[
a=\frac{1}{(q+1)^2},\qquad b=f^2 .
\]

\begin{lemma}[Lumping]\label{lem:lumping}
For $t\in T$ and a class $\mathcal K\in\{\mathcal E,\mathcal S,\mathcal O\}$, the sum $\sum_{u\in\mathcal K}M_{t,u}$ depends only on the class of $t$, and is the corresponding entry of
\[
\Mbar=\begin{pmatrix} bp & p-1 & b(v-p)\\ a(p-1) & bp & a(v-p)\\ ap & bp & a(v-p-1)\end{pmatrix},
\]
with rows and columns ordered $\mathcal E,\mathcal S,\mathcal O$.
\end{lemma}

\begin{proof}
Count successors using Lemma~\ref{lem:weights}. If $t\in\mathcal E$, then $\bar t$ is an extra. The other $p-1$ extras have weight $1$ and lie in $\mathcal S$; the $p$ special points have weight $f$ and lie in $\mathcal E$; the other $v-p$ points have weight $f$ and lie in $\mathcal O$. If $t\in\mathcal S$, then $\bar t$ is a special point: the $p$ extras have weight $f$, the other $p-1$ special points weight $1/(q+1)$, and the $v-p$ remaining points weight $1/(q+1)$. If $t\in\mathcal O$: the $p$ extras have weight $f$, the $p$ special points weight $1/(q+1)$, and the other $v-p-1$ non-special points weight $1/(q+1)$.
\end{proof}

The entry $p-1$ in the first row is the main danger of the design: from an extra, the walk may continue along any of the other extras with weight one.

\subsection{The decay estimate}

\begin{proposition}[Squared-word decay]\label{prop:decay}
Let $Q=q+1$ with $q=p^3$. For every $h\ge1$,
\[
\sum_{W\in T^h}P(W)^2\le(v+p^2+p)\Bigl(1-\frac{1}{4Q}\Bigr)^{h-1}.
\]
Consequently $\sum_{W\in T^h}P(W)^2\le e^{-2\delta h}$ with $\delta=1/(16Q)$ for all $h\ge 8Q\log(v+p^2+p)+2$.
\end{proposition}

\begin{proof}
Let $\varphi=(p+1,1,1)$ on $(\mathcal E,\mathcal S,\mathcal O)$ and $\theta=1-1/(4Q)$. We show $\Mbar\varphi\le\theta\varphi$ entrywise. Granting this, let $\tilde\varphi$ be the lift of $\varphi$ to $T$. By Lemma~\ref{lem:lumping}, $M\tilde\varphi\le\theta\tilde\varphi$; since $\mathbf 1\le\tilde\varphi$ and $M$ is nonnegative,
\[
\mathbf 1^{\top}M^{h-1}\mathbf 1\le\theta^{h-1}\,\mathbf 1^{\top}\tilde\varphi=\theta^{h-1}\bigl(p(p+1)+p+(v-p)\bigr)=\theta^{h-1}(v+p^2+p).
\]
We use three elementary facts, valid because $Q=p^3+1\ge28$:
\begin{enumerate}[label=(F\arabic*)]
\item $v=Q^2-Q+1>Q(Q-1)$, hence $b<(Q-1)^{-2}$ and $bv<Q/(Q-1)$;
\item $Q^2/(Q-1)^2\le(28/27)^2<1.08$;
\item $p^2\le Q/3$, since $p^3\ge3p^2$ for $p\ge3$.
\end{enumerate}

\emph{Row $\mathcal E$.} We have $(\Mbar\varphi)_{\mathcal E}=bp(p+1)+(p-1)+b(v-p)=(p-1)+bv+bp^2$. By (F1) and (F3),
\[
(\Mbar\varphi)_{\mathcal E}<(p-1)+1+\frac{1}{Q-1}+\frac{Q}{3(Q-1)^2}\le p+\frac{1}{27}+\frac{28}{3\cdot 27^2}<p+0.05 .
\]
On the other hand $\theta(p+1)=p+1-(p+1)/(4Q)$, and $(p+1)/(4Q)=1/\bigl(4(p^2-p+1)\bigr)\le1/28$. Hence $\theta(p+1)>p+0.96$.

\emph{Row $\mathcal S$.} Using $v=Q^2-Q+1$,
\[
(\Mbar\varphi)_{\mathcal S}=a(p^2-1)+bp+a(v-p)=1-\frac{Q+p-p^2}{Q^2}+bp .
\]
The inequality $(\Mbar\varphi)_{\mathcal S}\le\theta$ is equivalent to $Q/4+bpQ^2\le Q+p-p^2$. By (F1) and (F2), $bpQ^2<1.08p$, so it suffices that $p^2+0.08p\le 3Q/4$. This holds because $3Q/4\ge 9p^2/4$ by (F3).

\emph{Row $\mathcal O$.} Similarly
\[
(\Mbar\varphi)_{\mathcal O}=ap(p+1)+bp+a(v-p-1)=1-\frac{Q-p^2}{Q^2}+bp ,
\]
and $(\Mbar\varphi)_{\mathcal O}\le\theta$ follows from $p^2+1.08p\le3Q/4$, which again holds since $3Q/4\ge9p^2/4$ and $5p^2/4\ge1.08p$.

Finally, $(v+p^2+p)\theta^{h-1}\le(v+p^2+p)e^{-(h-1)/(4Q)}\le e^{-h/(8Q)}=e^{-2\delta h}$ as soon as $h\ge8Q\log(v+p^2+p)+2$.
\end{proof}

\begin{remark}[The exponent]\label{rem:exponent}
Two features of $\Mbar$ are close to critical. The $\mathcal O\to\mathcal O$ entry is about $1-1/q$, and the loop $\mathcal E\to\mathcal S\to\mathcal E$ has weight $(p-1)^2/(q+1)^2$. The loop must be small compared with the gap $1/q$, which requires $q$ to be large compared with $p^2$. For $q=p$ the spectral radius of $\Mbar$ exceeds $1$ (it is about $1.33$ for $p=3$). For $q=p^2$ it is below $1$ for the primes $p\le13$ we computed, but the test vector above fails; we have not pursued this case. Table~\ref{tab:radius} compares the actual spectral gap for $q=p^3$ with the gap $1/(4Q)$ proved above.
\end{remark}

\begin{table}[ht]
\centering
\begin{tabular}{@{}rrccc@{}}
\toprule
$p$ & $q=p^3$ & spectral radius $\rho(\Mbar)$ & $1-\rho(\Mbar)$ & $1/(4Q)$\\
\midrule
3 & 27 & 0.97646 & $2.35\times10^{-2}$ & $8.93\times10^{-3}$\\
5 & 125 & 0.99366 & $6.34\times10^{-3}$ & $1.98\times10^{-3}$\\
7 & 343 & 0.99751 & $2.49\times10^{-3}$ & $7.27\times10^{-4}$\\
11 & 1331 & 0.99932 & $6.83\times10^{-4}$ & $1.88\times10^{-4}$\\
13 & 2197 & 0.99958 & $4.20\times10^{-4}$ & $1.14\times10^{-4}$\\
\bottomrule
\end{tabular}
\caption{Spectral gap of the lumped matrix against the proved gap.}\label{tab:radius}
\end{table}

\section{Transfer of the geometric arguments}\label{sec:transfer}

We now explain how \cite[Sections~3--5]{OAI26} apply to the graphs of Section~\ref{sec:random}. Throughout, $x_A\in A^+$ and $x_B\in B$ are fixed before sampling.

\subsection{Bounded patterns}
We use the notion of a \emph{system} from \cite[Section~3]{OAI26}. It consists of finitely many positive-length paths in $\Gamma$, all cyclically immersed and closed except possibly one immersed interval path starting at $x_A$ or $x_B$. Disjoint occurrence intervals are paired in twos by length-preserving isometries; an order-preserving comparison requires equal letters and an order-reversing one inverse letters. The pairing is subject to the \emph{reduction condition}: paired occurrences traverse distinct undirected edges of $\Gamma$. Positions outside the paired intervals are \emph{unpaired}.

\begin{proposition}[Bounded-pattern estimate]\label{prop:bounded}
There is $\eps=\eps(p)>0$ with the following property. For all fixed $K\ge1$, $C\ge1$ and $I\ge0$, the probability that $\Gamma$ carries a system with at most $K$ paths, total length $H$ with $L\le H\le CL$, at most $I$ interval pairs and at most $\eps H$ unpaired occurrences tends to zero as $n\to\infty$.
\end{proposition}

\begin{proof}
We follow the proof of \cite[Proposition~3.1]{OAI26} and indicate the changes.

\emph{Inputs.} Proposition~\ref{prop:decay} replaces \cite[Lemma~2.1]{OAI26}, with $\delta=1/(16(q+1))$. By Lemma~\ref{lem:weights}, every immersed word $W$ of length $h$ satisfies $P(W)\ge(q+1)^{-(h-1)}$, so \cite[(3.9)]{OAI26} holds with $a_0=2\log(q+1)$. Corollary~\ref{cor:prescriptions} replaces \cite[(2.7)]{OAI26}.

\emph{Signs in the patterns.} Let $\Delta$ be the image graph of a candidate system and $\Delta^A$ the union of its components on side $A$. We enlarge each unlabelled pattern of \cite{OAI26} by a $2$-colouring $\chi$ of $\Delta^A$, recording which vertices are to be realized in $A^+$. Since $\Delta$ has boundedly many components, this multiplies the number of patterns by a bounded factor. A pattern for which $\Delta^A$ contains an odd cycle cannot be realized in the bipartite graph $\Gamma_A$ and has probability zero.

\emph{The realization count.} Fix a stage $j$ and a permitted labelling, as in \cite[(3.10)]{OAI26}. Write $E^Y_j,V^Y_j,Q^Y_j$ for the numbers of edges, vertices and active chains of $\Delta_j$ on side $Y\in\{A,B\}$, and $k^Y_j\le1$ for the number of fixed root vertices. On side $A$, a marked vertex of prescribed colour has at most $n/2$ images, and an internal chain vertex with incoming letter $t$ and outgoing letter $u$ has at most $(n/2)f\,w(t,u)$ images, since $A^+$ and $A^-$ carry identical type proportions. On side $B$ the corresponding bounds are $n$ and $nf\,w(t,u)$. By Corollary~\ref{cor:prescriptions}, the $E^A_j+E^B_j$ distinct prescriptions cost at most $(2/(nf))^{E^A_j}(1/(nf))^{E^B_j}e^{o(L)}$. Multiplying, the expected number of realizations is at most
\[
2^{\,E^A_j-V^A_j+k^A_j}\;n^{\,V_j-E_j-k_j}\,f^{-Q_j}\prod w(t,u)\;e^{o(L)},
\]
where the product runs over internal chain turns. Apart from the first factor this is the bound of \cite[(3.10)]{OAI26}. Now $E^A_j-V^A_j$ is at most the cycle rank of the $A$-part of $\Delta_j$, which is at most the cycle rank of $\Delta$, and \cite[Section~3]{OAI26} shows that the latter is bounded in terms of $C$ and $K$. The first factor is therefore bounded and is absorbed into $e^{o(L)}$.

\emph{Conclusion.} The remaining steps of \cite[Section~3]{OAI26} --- multiplicity stages and the vertex-exponent inequality \cite[(3.4)]{OAI26}, aligned blocks, the treatment of self-links, binning, and stage selection --- concern traversal counts and words, not types, and apply verbatim. The resulting constant is $\eps=\min\{1/4,\ \delta/(16(1+a_0))\}$, which depends only on $p$, so the order of quantifiers is as in \cite{OAI26}.
\end{proof}

\subsection{Spherical arrangements}
A \emph{reduced spherical arrangement over $\Gamma$} is defined as in \cite[Definition~4.1]{OAI26}.

\begin{proposition}\label{prop:noarrangement}
With probability $1-o(1)$, $\Gamma$ admits no reduced spherical arrangement. In particular, for all sufficiently large $n\in2v^2\Z$ there is an outcome of Construction~\ref{con:matchings} with $\girth(\Gamma)\ge L$, the diameter bound of Lemma~\ref{lem:diameter}, and no reduced spherical arrangement.
\end{proposition}

\begin{proof}
The proof of \cite[Proposition~4.5]{OAI26} uses three probabilistic inputs: the bounded-pattern estimate, the diameter bound, and the closing lemma \cite[Lemma~4.2]{OAI26}. We supply the first two by Proposition~\ref{prop:bounded} and Lemma~\ref{lem:diameter}. The closing lemma needs only girth at least $3$ and minimum degree at least $3$ after deleting two edges at a vertex; here the minimum degree is $q+1\ge28$. The remaining ingredients, \cite[Lemmas~4.3 and 4.4]{OAI26}, concern planar arrangements and planar graphs; they use only immersion and girth, which hold here. The constants are chosen in the order of \cite[(4.5)]{OAI26}, starting from the $\eps$ of Proposition~\ref{prop:bounded}.
\end{proof}

\subsection{The cone complex}
Let $F$ be the rose with $(v+p)/2$ edges. For each component $\Lambda$ of $\Gamma$, attach the abstract cone $K_\Lambda$ on $\Lambda$ to $F$ along the labelled map $\Lambda\to F$. Let $X$ be the resulting finite two-dimensional complex and $G_p=\pi_1(X)$, which is finitely presented. The labels of all closed walks in $\Gamma$ are trivial in $G_p$.

\begin{proposition}\label{prop:topology}
If $\Gamma$ admits no reduced spherical arrangement, then:
\begin{enumerate}[label=\textup{(\roman*)}]
\item the universal cover of $X$ is contractible, so $X$ is a finite two-dimensional classifying space for $G_p$ and $G_p$ is torsion-free;
\item every path in $\Gamma$ from $x_A$ or $x_B$ to a different vertex has nontrivial label in $G_p$.
\end{enumerate}
\end{proposition}

\begin{proof}
This is \cite[Proposition~5.1 and Corollary~5.3]{OAI26}. Their proofs use only the absence of reduced spherical arrangements together with general topology: cone pictures, minimization of boundary length, band surgery, and the Hurewicz and Whitehead theorems \cite{Hat02}. Torsion-freeness follows from the finite-length free resolution given by the cellular chains of the universal cover \cite{Bro82}.
\end{proof}

\section{Proof of Theorem~\ref{thm:main}}\label{sec:proof}

Fix an odd prime $p$, put $q=p^3$, and fix $x_A\in A^+$ and $x_B\in B$.
\begin{enumerate}[label=\arabic*.]
\item By Proposition~\ref{prop:noarrangement}, for some large $n\in2v^2\Z$ there is an outcome of Construction~\ref{con:matchings} with $\girth(\Gamma)\ge L$ that admits no reduced spherical arrangement. Fix it.
\item Build $X$ and $G_p=\pi_1(X)$ as in Section~\ref{sec:transfer}. By Proposition~\ref{prop:topology}, $G_p$ is finitely presented and torsion-free, $X$ is a finite two-dimensional classifying space for it, and root separation holds.
\item Let $A'$ and $B'$ be the vertex sets of the components of $x_A$ and $x_B$, and let $\phi\colon\mathbb F(T)\to G_p$ be the natural map. Define $\alpha$ and $\beta$ as in Lemma~\ref{lem:signed}, with $\eps=+1$ on $A^+$ and $\eps=-1$ on $A^-$. Hypothesis (i) holds because $\Gamma_A$ is bipartite, and hypothesis (ii) holds with $d=1$ by Lemma~\ref{lem:residue}. Hence $\alpha\beta=0$ in $\F_p[G_p]$.
\item By root separation, $x_A$ is the only vertex of $A'$ with $g_x=1$, and $x_B$ the only vertex of $B'$ with $h_y=1$. The identity coefficients of $\alpha$ and $\beta$ are therefore $\eps(x_A)=1$ and $1$, so $\alpha\neq0$ and $\beta\neq0$.\qed
\end{enumerate}

\section{Further remarks}\label{sec:remarks}

\subsection{Characteristic two}
For $p=2$, Lemma~\ref{lem:signed} is the cancellation of \cite{OAI26}. Our design does not specialize to $p=2$: with two extras, $\lvert T\rvert=v+2$ would be odd and no fixed-point-free involution would exist. The asymmetric three-extra design of \cite{OAI26} avoids this.

\subsection{Characteristic zero}
Over $\Z$, Step~3 of Lemma~\ref{lem:signed} shows that the contribution of a component $C$ to $\alpha\beta$ is $\bigl(\sum_{(x,y)\in C}\eps(x)\bigr)g_C$, and that $d$ times the signed count vanishes if every degree in $C$ equals $d$ exactly. A version of the construction with $\lvert S_x\cap S_y\rvert=d$ for all $x\in A$, $y\in B$ would therefore give zero divisors over $\Q$. Families with constant cross-intersections exist, but the obvious ones destroy the decay estimate: if, for instance, every type on $B$ equals one set $S$, then $w(t,u)=1$ for all non-cancelling turns inside $S$, and $\sum_WP(W)^2$ grows exponentially. We do not know a design that combines exact constancy with squared-word decay. A rational counterexample of any kind would contradict the strong Atiyah conjecture for the group concerned, since that conjecture implies the zero-divisor conjecture over $\C$ for torsion-free groups \cite{Lin93,Lue02}.

\subsection{Uniformity in $p$}
The groups $G_p$ depend on $p$, so no single group is shown to violate the conjecture in two different characteristics. If one fixed group $G$, with fixed finite supports for $\alpha$ and $\beta$, admitted solutions over $\F_p$ for infinitely many $p$, one could normalize one coefficient of each of $\alpha$ and $\beta$ to $1$ along an infinite set of primes, and elimination theory (the Nullstellensatz applied to the resulting system of polynomial equations over $\Z$) would give a solution over $\overline{\Q}$, though not necessarily over $\Q$. Since $\overline{\Q}\subseteq\C$, this would give zero divisors in $\C[G]$. Our construction provides no such uniformity.

\subsection{Explicitness}
As in \cite{OAI26}, the argument is probabilistic. It proves that suitable matchings exist for large $n$ but does not produce a presentation of $G_p$ or the elements $\alpha,\beta$ explicitly. The alphabet alone has $760$ letters for $p=3$.

\subsection{Unique products}
Since $\F_p[G_p]$ has zero divisors, $G_p$ does not have the unique-product property.

\appendix

\section{A check of the arguments of \cite{OAI26}}\label{app:verification}

Our proof relies on Sections~3--5 of \cite{OAI26} and on the lemmas of its Section~2 that they use. We went through every numbered statement, displayed formula and inference in these sections, re-derived each step, and recomputed the constants exactly. We found no error that affects the conclusions. The following six points deserve a clarifying sentence; none changes a statement or a constant that matters. This check was carried out with the assistance of an AI model (see the Acknowledgements). 

\begin{remark}[Uniformity in the switching lemma]
In \cite[Lemma~2.2]{OAI26} the error $r_n$ counts the edges meeting the $2L$-neighbourhood of $\{x,y\}$, plus one. It is therefore $O(d_*^{2L+2})$ for every set of prescriptions, as claimed. The injectivity of the switching map follows because the output determines the switched edge: the new neighbour of $x$ along $t$ is $z$, and the new neighbour of $y$ along $\bar t$ is $u$.
\end{remark}

\begin{remark}[Joins of block occurrences]
In the proof of \cite[Proposition~3.1]{OAI26}, the claim that joins form a partial matching of block occurrences uses $\kappa<1/2$, so that a compared substring occupies more than half of its block. This holds for large $L$ because $\kappa\to0$, but it is not stated.
\end{remark}

\begin{remark}[Vertex images on side $A$]
The bound $np\,w(t,u)$ for the number of images of an internal chain vertex, used in \cite[(3.10)]{OAI26}, needs the conditional frequencies on side $A$ to be at most the weights computed on side $B$. Entrywise, the only differing entry is the extra--extra one, which is $1/2$ on $A$ against $1$ on $B$.
\end{remark}

\begin{remark}[Same-arc digons]
In \cite[Lemma~4.3]{OAI26}, a complementary disk digon whose two sides come from the same arc requires both ends of that arc to lie on disks with a single arc end, hence with a boundary of length one. A cyclically immersed closed path of length one is a loop, which girth excludes, and at most one of the two disks can be the exceptional one. So the case is vacuous. The coarse bound $2N$ is harmless, and \cite[(4.2)]{OAI26} in fact gives $B_{\mathrm{gap}}\le 8N-8$.
\end{remark}

\begin{remark}[Splitting long paths]
In the proof of \cite[Proposition~4.5]{OAI26}, a path of length $h>UL$ is split into $m=\lceil h/(UL)\rceil$ segments. Then $m\le 2h/(UL)$, so the stated bound $m\le3h/(UL)$ is loose but correct.
\end{remark}

\begin{remark}[Exterior surgery]
In \cite[Section~5.3]{OAI26}, cases 2 and 4 modify the exterior disk. In each row of the splice table the new outer path joins the original distinct endpoints, and the retained domain is a disk: in case 2 one removes from the domain disk an inner disk attached to its boundary by a band, and in case 4 one cuts the domain disk along a band with both ends on its boundary and keeps one of the two pieces. The boundary length drops by two before tightening.
\end{remark}

We also recomputed the arithmetic of \cite{OAI26}. The two row bounds in \cite[Lemma~2.1]{OAI26} are $0.993152\ldots$ and $0.601702\ldots$, both below $149/150$, and the threshold $h\ge3002$ is more than sufficient (the argument needs $h\ge2897$). The vertex counts of \cite[Section~2.1]{OAI26}, the inequalities \cite[(4.4)]{OAI26}, and the surviving proportion $5/16$ in the proof of \cite[Proposition~4.5]{OAI26} are all correct.

\section{Computations}\label{app:computations}

The following checks were run in Python with exact rational arithmetic where stated.

\begin{enumerate}[label=(\arabic*)]
\item \emph{Proposition~\ref{prop:decay}.} The inequality $\Mbar\varphi\le\theta\varphi$ was verified exactly for every odd prime $p<1000$ and $q\in\{p^3,p^4,p^5\}$.
\item \emph{Lumping.} For $p=3$ we constructed $\PG(2,27)$ from a model of $\F_{27}$, read the turn weights off its incidences (obtaining exactly the values $1/28$, $28/757$ and $1$), and formed the full $760\times760$ matrix $M$. Its spectral radius is $0.976461$, equal to that of $\Mbar$.
\item \emph{Signed cancellation.} For $p=q=3$ we built random graphs of the shape of Constructions~\ref{con:types} and~\ref{con:matchings} with $338$ vertices on each side. These graphs have small girth and give no torsion-free group; they test only the counting of Lemma~\ref{lem:signed}. The step graph has $28{,}245$ components. All but one have size $\equiv2\pmod 3$, so the unsigned count fails modulo $3$. Every component satisfies the identity of Step~3 exactly, and every signed count is divisible by $3$. The three largest components have sizes $38{,}808$, $131$ and $113$, with signed counts $0$, $3$ and $-3$.
\end{enumerate}

\section*{Acknowledgements}
This work benefited substantially from assistance by Claude, an AI model developed by Anthropic. Claude contributed to the development of the signed-cancellation argument and the type design in Sections~\ref{sec:signed}--\ref{sec:decay}, assisted with the computations in Appendix~\ref{app:computations} and the reading of \cite{OAI26} summarized in Appendix~\ref{app:verification}, and helped draft the manuscript. Claude is not an author and cannot assume responsibility for the work; responsibility for the content rests with the author. We thank the referee for helpful comments.


\begin{thebibliography}{99}

\bibitem{Ber19}
V.~Bereznyuk, \emph{Asphericity of groups defined by graphs}, Math. Notes \textbf{105} (2019), 316--328. \href{https://arxiv.org/abs/1711.05293}{arXiv:1711.05293}.

\bibitem{Bro82}
K.~S. Brown, \emph{Cohomology of Groups}, Graduate Texts in Mathematics, vol.~87, Springer, New York, 1982.

\bibitem{FSP26}
S.~P. Fisher and P.~S\'anchez-Peralta, \emph{Division rings for group algebras of virtually compact special groups and $3$-manifold groups}, J. Comb. Algebra \textbf{10} (2026), no.~1/2, 153--193.

\bibitem{Gar21}
G.~Gardam, \emph{A counterexample to the unit conjecture for group rings}, Ann. of Math. (2) \textbf{194} (2021), no.~3, 967--979.

\bibitem{Gar24}
G.~Gardam, \emph{Non-trivial units of complex group rings}, preprint, 2024. \href{https://arxiv.org/abs/2312.05240v2}{arXiv:2312.05240v2}.

\bibitem{GM25}
M.~Garg and I.~Mineyev, \emph{On zero-divisors and units in group rings of torsion-free CAT(0) groups}, preprint, 2025. \href{https://arxiv.org/abs/2501.07646v2}{arXiv:2501.07646v2}.

\bibitem{Gro03}
M.~Gromov, \emph{Random walk in random groups}, Geom. Funct. Anal. \textbf{13} (2003), no.~1, 73--146.

\bibitem{Gru15}
D.~Gruber, \emph{Groups with graphical $C(6)$ and $C(7)$ small cancellation presentations}, Trans. Amer. Math. Soc. \textbf{367} (2015), no.~3, 2051--2078.

\bibitem{Hat02}
A.~Hatcher, \emph{Algebraic Topology}, Cambridge University Press, Cambridge, 2002.

\bibitem{Hig40a}
G.~Higman, \emph{Units in group rings}, D.Phil. thesis, University of Oxford, 1940.

\bibitem{Hig40b}
G.~Higman, \emph{The units of group-rings}, Proc. London Math. Soc. (2) \textbf{46} (1940), 231--248.

\bibitem{Kap57}
I.~Kaplansky, \emph{Problems in the theory of rings}, in: Report of a Conference on Linear Algebras, June 6--8, 1956, National Academy of Sciences--National Research Council Publication 502, Washington, DC, 1957, pp.~1--3.

\bibitem{KLM88}
P.~H. Kropholler, P.~A. Linnell, and J.~A. Moody, \emph{Applications of a new $K$-theoretic theorem to soluble group rings}, Proc. Amer. Math. Soc. \textbf{104} (1988), no.~3, 675--684.

\bibitem{Lin93}
P.~A. Linnell, \emph{Division rings and group von Neumann algebras}, Forum Math. \textbf{5} (1993), no.~6, 561--576.

\bibitem{LT79}
R.~J. Lipton and R.~E. Tarjan, \emph{A separator theorem for planar graphs}, SIAM J. Appl. Math. \textbf{36} (1979), no.~2, 177--189.

\bibitem{Lue02}
W.~L\"uck, \emph{$L^2$-Invariants: Theory and Applications to Geometry and $K$-Theory}, Ergebnisse der Mathematik und ihrer Grenzgebiete (3), vol.~44, Springer, Berlin, 2002.

\bibitem{MS26}
I.~Mian and S.~Siddique, \emph{A machine-checked proof that Gardam's $\tilde A_2$ lattice does not have unique products}, preprint, 2026. \href{https://arxiv.org/abs/2609.22380v1}{arXiv:2609.22380v1}.

\bibitem{Min24}
I.~Mineyev, \emph{The topology and geometry of units and zero-divisors: Origami}, preprint, 2024, revised 2026.

\bibitem{OAI26}
OpenAI, \emph{A torsion-free group algebra with zero divisors}, preprint, September 23, 2026.

\bibitem{Oll05}
Y.~Ollivier, \emph{A January 2005 Invitation to Random Groups}, Ensaios Matem\'aticos, vol.~10, Sociedade Brasileira de Matem\'atica, Rio de Janeiro, 2005.

\bibitem{Oll06}
Y.~Ollivier, \emph{On a small cancellation theorem of Gromov}, Bull. Belg. Math. Soc. Simon Stevin \textbf{13} (2006), no.~1, 75--89.

\bibitem{RS87}
E.~Rips and Y.~Segev, \emph{Torsion-free group without unique product property}, J. Algebra \textbf{108} (1987), no.~1, 116--126.

\bibitem{San81}
R.~Sandling, \emph{Graham Higman's thesis ``Units in group rings''}, in: Integral Representations and Applications (Oberwolfach, 1980), Lecture Notes in Mathematics, vol.~882, Springer, 1981, pp.~93--116.

\bibitem{Shi26}
H.~Shin, \emph{A global girth obstruction for Garg--Mineyev taiko product structures}, preprint, 2026. \href{https://arxiv.org/abs/2607.01716v1}{arXiv:2607.01716v1}.

\bibitem{Ste15}
M.~Steenbock, \emph{Rips--Segev torsion-free groups without the unique product property}, J. Algebra \textbf{438} (2015), 337--378.

\end{thebibliography}
\end{document}